\documentclass[a4paper,twosided, 11pt]{article}
\usepackage[utf8]{inputenc} 
\usepackage[T1]{fontenc} 
\usepackage[a4paper]{geometry} 
\usepackage{amsmath, amssymb, mathtools, amsthm}
\usepackage{mathtools} 
\usepackage{mathrsfs}  
\usepackage{graphicx} 
\usepackage{xcolor}   
\usepackage[english]{babel} 
\usepackage{bbm}
\usepackage[title]{appendix}
\newcommand{\eps}{\varepsilon}
\usepackage{caption}
\usepackage{subcaption}
\usepackage{comment}
\usepackage{geometry}
\usepackage{cite}
\usepackage{empheq}
\usepackage{float}
\usepackage{booktabs}
\usepackage{varwidth}
\usepackage{enumitem} 
\usepackage[colorlinks=true]{hyperref} 

\newcommand \commentout[1] {}

\newcommand{\R}{\mathbb{R}}
\newcommand{\N}{\mathbb{N}}

\newcommand {\e}  {\varepsilon}

\newcommand {\Chi} {{\bf \raise 2pt \hbox{$\chi$}} }

\newcommand {\Div}  { {\rm div} }

\newcommand*{\dd}{\mathop{\kern0pt\mathrm{d}} {}}

\newcommand*{\DD}{\mathop{\kern0pt\mathrm{D}} {}}

\DeclareMathOperator*{\supp}{\operatorname{supp}}

\theoremstyle{plain}
\newtheorem*{thm*}{Theorem}
\newtheorem{thm}{Theorem}[section]

\newtheorem{lemma}[thm]{Lemma}

\newtheorem{proposition}[thm]{Proposition}
\newtheorem{corollary}[thm]{Corollary}

\theoremstyle{remark}
\newtheorem{remark}[thm]{\bf Remark}
\newtheorem{definition}[thm]{\bf Definition}
\newtheorem{example}[thm]{\bf Example}
\newcommand{\ie}{\emph{i.e.}\;}

\newcommand{\one}{\mathbf{1}}

\newcommand{\beq}{\begin{equation}}
\newcommand{\eeq}{\end{equation}}
\newcommand{\bea} {\begin{array}{rl}}
\newcommand{\eea} {\end{array}}
\newcommand{\bepa}{\left\{ \begin{array}{l}}
\newcommand{\eepa} {\end{array}\right.}
\newcommand{\diff}{\mathop{} \mathrm{d}}

\numberwithin{equation}{section}
\usepackage{tikz}
\usetikzlibrary{arrows.meta,decorations.pathreplacing,calc}

\title{Interfaces and non-uniqueness in a cross-diffusion system with independent drifts}

\author{
    Charles Elbar%
     \thanks{Université Claude Bernard Lyon 1, ICJ UMR5208, CNRS, Ecole Centrale de Lyon, INSA Lyon, Université Jean Monnet, 69622
Villeurbanne, France. Email: elbar@math.univ-lyon1.fr; parker@math.univ-lyon1.fr} %
    \and 
    Guy Parker \footnotemark[1]
}
\date{}

\begin{document}

\maketitle

\begin{abstract}

We study a one-dimensional cross-diffusion system of two populations. Their densities are diffused with a common pressure that depends on the
total density, but are transported by two independent
external potentials. Starting from segregated initial data with a
single interface, we show that the way the two densities meet is not given by the equation alone: it depends on the notion of solution one chooses. When the drifts push the two phases towards each other at the interface, the vanishing-viscosity solution creates an overlap, whereas a segregated weak solution also exists. The two solutions are distinct,
so the Cauchy problem is not well posed in the class of weak solutions.
We first prove the result for an explicit stationary segregated solution, and then extend the non-uniqueness to general segregated data.
\end{abstract}
\vskip .7cm

2020 \textit{Mathematics Subject Classification.} 35K55, 35K65, 35A02, 35D30.\newline\textit{Keywords and phrases.} cross-diffusion, non-uniqueness, vanishing viscosity, segregated solutions, relative entropy, free boundaries.

\section{Introduction}

 We consider the evolution of two densities $\rho(t,x)$ and $\mu(t,x)$ on the whole space $\R$ over a time interval $[0,T]$ for some $T>0$. 
 They evolve according to the cross-diffusion system:

\begin{align}
\partial_t\rho - \partial_x(\rho\,\partial_x f'(\rho+\mu))- \partial_x (\rho\,\partial_x V) &= 0,\label{eq:rho}\\
\partial_t\mu  - \partial_x (\mu\,\partial_x f'(\rho+\mu))- \partial_x (\mu\,\partial_x W) &= 0.\label{eq:mu}
\end{align}

Here, the system is equipped with non-negative initial densities $\rho(0,\cdot) = \rho_0$ and $\mu(0,\cdot) = \mu_0$. 
The functions $V,W\in C^3(\R)$ are external potentials, and the non-linearity $f':\R\to\R$ is the pressure law.

\medskip 

We are motivated by the following phenomenon: suppose the two populations are initially segregated, that is $\rho_0$ is supported
on $\{x\le 0\}$ and $\mu_0$ on $\{x\ge 0\}$, and they meet at the single interface $x=0$. A natural question is whether the segregation persists. If the two drifts coincide, that is $V\equiv W$, then the common velocity simply transports the interface, and the two densities stay segregated. The interesting regime is when the drifts are different and, in particular, when they push the two phases towards one another at the interface. 

\medskip 

The velocity of $\rho$ is $-\partial_x f'(\rho+\mu)-\partial_x V$ and that of $\mu$ is $-\partial_x f'(\rho+\mu)-\partial_x W$. 
At the interface, the pressure terms are the same, so the motion of the two phases depends only on the difference of the drifts. 
The two phases approach each other exactly when
\begin{equation}\label{eq:inward}
\partial_xW(0)-\partial_x V(0)>0 .
\end{equation}

At this point, we show that the notion of solution matters. One may select solutions by vanishing viscosity, adding $\eps\,\partial_{xx}$ to each equation and letting $\eps\to 0$. 
On one hand, this is the natural approach that usually selects a physical solution and is the one for which existence is usually done. On the other hand, one may obtain a weak solution that keeps the two phases segregated. We show that, under 
\eqref{eq:inward}, these two approaches yield different solutions, hence showing non-uniqueness.

\subsection{Main ideas}

Our results are of two kinds.
\begin{itemize}
    \item \textit{An explicit example.} For suitable convex $f$ and potentials $V,W$ there exists an explicit stationary state $(\bar\rho,\bar\mu)$ in which the gradient of the pressure is equal to each drift on its own side of the interface. We then prove that no vanishing-viscosity solution can be segregated when \eqref{eq:inward} holds. In particular, $(\bar\rho,\bar\mu)$ is not the vanishing-viscosity solution, and the two are distinct weak solutions with the same data.
    \item \textit{General data.} We then show that the phenomenon is not an artefact of the explicit profile. For the power-type pressures
    $$
    f_\alpha'(s)=\frac{\alpha}{\alpha-1}s^{\alpha-1}\quad(\alpha\neq1),
    \quad f_1'(s)=\log s,
    $$
with $\alpha>\frac13$, and for any segregated initial data, the segregated
solution and the vanishing-viscosity solution are again different.
\end{itemize}

The method behind both results is the relative-entropy functional
\begin{equation}\label{eq:Phi-intro}
\Phi[\rho,\mu]=\int_\R h(\rho,\mu)\diff x,
\quad h(u,v)=(u+v)\log(u+v)-u\log u-v\log v .
\end{equation}
 Here $h$ is non-negative and vanishes exactly where $uv=0$, so $\Phi$ measures the overlap of the two phases: $\Phi=0$ iff the state is segregated. The main identity, which can be made at the viscous level and sending $\eps\to 0$ is the lower bound
$$
\Phi[\rho(t),\mu(t)]\ge \int_0^t \langle q(s),\partial_x V-\partial_x W\rangle\diff  s,
$$
where $q$ is a distribution that coincides with a (negative) Dirac mass for densities with a shared interface. Thus the right-hand side is strictly positive under \eqref{eq:inward}. A segregated solution has $\Phi\equiv 0$ and therefore cannot satisfy this bound.

\subsection{Related works}
\emph{An Overview:} In the absence of external potentials $\ie$ when $V \equiv W \equiv 0$, System \eqref{eq:rho}--\eqref{eq:mu} was simultaneously introduced in \cite{TravisBusenberg83}, as an epidemiological model, and in \cite{GurtinPipkin84}, as a means of describing crowd motion. 
As a novelty in the study of cross-diffusive systems at the time, these works proposed that the diffusivity of each sub-population should not depend on the density of the sub-population itself but, rather, on the population density as an aggregate. 
This was in accordance with the rationale proposed in \cite{gurtinmaccamy77}. 

\medskip 

Following the formal construction of the system, weak well-posedness was proven, in one dimension, in the works \cite{BertschPeletier85, BertschGurtin87EqualVelocity}. 
A striking property, established in \cite{BertschPeletier85}, is that initially segregated sub-populations remain segregated. 
That is to say, if we consider a one-dimensional solution $(\rho,\mu)$ of the below: System \eqref{eq:rholit}--\eqref{eq:Plit}, for which we assume that $\rho_0,\mu_0$ are densities of disjoint support, then $\rho_t, \mu_t$ must possess essentially disjoint support for all $t > 0$.

\medskip

\begin{align}
\partial_t\rho &- \Div(\rho\,\nabla P)= 0,\label{eq:rholit}\\
\partial_t\mu  & - \Div(\mu\,\nabla P)= 0,\label{eq:mulit}\\
P & = f'(\rho+\mu).\label{eq:Plit}
\end{align}

\medskip 

Due to the segregation phenomenon, a key feature of solutions is that densities can produce jump discontinuities at sharp inter-species interfaces and, consequent to this mechanism, System \eqref{eq:rholit}--\eqref{eq:Plit} has since become ubiquitous in the modelling of heterogeneous cell populations and tumour growth as motivated in \cite{CarrilloMurakawa19,FuGrietteMagal20} and \cite{chaplain2006mathematical,jungel2026multiphasecrossdiffusionmodelstissue} respectively.
Indeed, in the presence of reaction terms, which facilitate the modelling of growth and death dynamics, the literature and understanding of System \eqref{eq:rholit}--\eqref{eq:Plit} is rich.
In particular, the existence of solutions in the presence of species-heterogeneous reaction functions was proven in \cite{CarrilloFagioli18,DruetJungel20}; whilst, under more homogeneous assumptions on the reaction terms, a higher-dimensional existence and regularity theory was established due to \cite{BertschHilhorst12,GwiazdaPerthame19,PriceXu20,Jacobs23CrossDiff}.
Furthermore, the regularity of the free-boundary problem was studied in \cite{KimTon:21}; whilst \cite{BertschHilhorst15,CarrilloLorenzi} provided an existence and analysis of travelling wave solutions.

\medskip 

The presence of sharp-interspecies interfaces has also motivated the study of the incompressible limit of System \eqref{eq:rholit}-\eqref{eq:Plit} in the presence of various growth mechanisms.
Specifically, studies have been produced for several models which utilise Darcy's law \cite{ChertockDegondHecht19,DegondHectVauchelet20,BubbaPerthame20,LiuXu21,David23,MR4745662}. 
The Darcy law describing the constitutive velocity-pressure relation given by
\[
\boldsymbol{v} = -\nabla P
\]
which we consider in this manuscript.

\medskip 

\emph{Drift-Diffusion Systems:} Despite the wealth of literature concerning reaction-diffusion variants of System \eqref{eq:rholit}--\eqref{eq:Plit}, it is only more recently that the existence theory has encompassed the drift-diffusion variant: System \eqref{eq:rho}--\eqref{eq:mu}, allowing for the expression of drifts (or aggregation effects) which may differ among the two species. 
The motivation for such a generalisation is that, alongside the diffusive dynamics, each species should be allowed to express a heterogeneous bias within the environment (cf. \cite{BurgerDifrancesco18,CarrilloHuang18}).
This said, the first order heterogeneity induced by the drifts generates significant difficulty in establishing an existence theory.

\medskip

In this direction, \cite{MeszarosKim2018} established a weak existence theory by assuming an ordering on the drifts along with a well chosen ordering on the initial datum.
This facilitated a first weak existence result, yielding segregated solutions in one dimension. 
The theory has since seen the construction of weak solutions in one spatial dimension, allowing for the mixing of the two densities, without ordering assumptions, and for a general class of potentials.
In particular, the work \cite{2025arXiv250418484M} established weak existence for System \eqref{eq:rho}--\eqref{eq:mu} under a logarithmic pressure law, assuming initial data with joint support. 
This result relied on $L^\infty_tBV_x$ estimates and was later extended in \cite{ElbarSantambrogio25} to allow for fast-diffusive pressure laws and in \cite{MeszarosParker26} to facilitate segregated, mixed or partially mixed data.
Simultaneously, the weak existence of solutions in the absence of any $BV$ regularity was established in \cite{Skrzeczkowski2026crossdiffusion} by using a compensated compactness approach, allowing for low regularity initial data and pressure laws of slow or fast diffusion type.

\medskip 

It is noteworthy that each of the results which assert existence for mixed initial profiles \cite{2025arXiv250418484M, ElbarSantambrogio25,MeszarosParker26,Skrzeczkowski2026crossdiffusion} have been produced by establishing a sufficient compactness for a smooth, continuous in time, approximation of System \eqref{eq:rho}--\eqref{eq:mu}. 
On the other hand, \cite{MeszarosKim2018} established the existence of segregated solutions by means of a time discretised minimising movement scheme, exploiting that System \eqref{eq:rho}--\eqref{eq:mu} formally corresponds to the $2$-Wasserstein gradient flow of the free energy 
\begin{equation}\label{eq:freeenergy}
 \mathcal{F}[\rho,\mu] \coloneqq \int_{\R} f(\rho+ \mu) + V \rho + W \mu \diff x.   
\end{equation}
This dichotomy was further exemplified in the recent work
\cite{santambrogioschulz2026segregated} where, also utilising a minimising movement scheme, the one-dimensional well-posedness theory for segregated solutions was extended to a general class of drifts and segregated initial data, in particular, without any ordering assumptions. 

\medskip 

In adjacency with System \eqref{eq:rho}--\eqref{eq:mu}, there exists a wealth of literature concerning cross-diffusive systems which feature a pressure law dependent on the aggregate density but which also feature additional diffusive terms (see, for instance, \cite{LaurencotMatioc13,DiFrancescoEspositoFagioli18,Laborde20,HirvonenJungel2026}).
Typically, these diffusive terms produce self-diffusive behaviour in at least one of the sub-species, leading to a drastic change in the dynamics.
Specifically, concerning the study of the behaviour these self-diffusive systems as they approach the degenerate parabolic-hyperbolic limit, we mention the work \cite{AlasioBrunaFagioliSchulz22}, which demonstrates how the system dynamics change as the self-diffusive parameter degenerates in the presence of independent drift terms, and also the work \cite{BurgerCarrillo}, which establishes the emergence of segregatory behaviour in the singular limit. 
We also mention the work \cite{CarilloChenBangJungel25} which provides a study of the singular limit of the fluid dynamical approximation when $V \equiv W \equiv 0$.

\medskip 

To be precise, a crucial juxtaposition between System \eqref{eq:rho}--\eqref{eq:mu} and systems which exhibit self-diffusive dynamics is the higher regularity of each sub-species undergoing self-diffusion. 
In particular, solutions possessing Sobolev regularity cannot produce sharp inter-species interfaces in the way that one expects in the singular limit.
Consequently, we do not expect that the techniques which apply in this manuscript may be used to establish a non-uniqueness result for self-diffusive systems.
Similarly, we expect that the interface instability phenomenon is limited strictly to cross-diffusion systems of $N \geq 2$ species and, consequently, such behaviour is not expected to occur for generic single species advection-diffusion models such as those studied in \cite{KimZhang18}.

\medskip 
\emph{Non-Uniqueness for Systems:} 
It is well known that systems of uniformly parabolic equations can exhibit wild behaviour, leading to the global ill-posedness of solutions, even in the presence of smooth or analytic coefficients (cf. \cite{Struwe84,StaraJohn95}). 
Furthermore, when the data is chosen to be singular enough, non-uniqueness may even develop for a linear elliptic equation of a single species (see \cite{Serrin64,Prignet95}, see also  \cite{MullerRieger05} for parabolic systems). 
Consequently, a stronger notion of solution, such as an entropy solution \cite{Carrillo99Entropy}, or renormalised solution \cite{BlanchardMurat97}, is often necessary to recover a full well-posedness theory. 
However, despite such classical counterexamples, the literature exposing non-uniqueness for cross-diffusion systems is sparse and, general, pathological counterexamples may appear divorced from the practical existence theory.

\medskip 

The contribution of this work is to establish the non-uniqueness of solutions for System \eqref{eq:rho}--\eqref{eq:mu}, precisely by showing that sharp-interspecies interfaces can spontaneously breakdown, leading to a zone of mixing. 
To the best of the authors' knowledge, this result is the first of its kind amongst such models.
Moreover, this phenomenon has not been observed in the context of cross-diffusion systems but, rather, bears a striking resemblance to the non-uniqueness behaviour observed for the Muskat problem \cite{Muskat34}.

\medskip 

The Muskat problem considers the evolving interface of two fluids of variable density which initially sit in segregation, one atop the other. 
Consequent to the gravitational forces acting on the system, it has been shown that the fluid produces an unstable interface when the denser fluid is initialised on top of the profile. 
Moreover, as was established throughout \cite{Otto99,CastroCordoba21,noisetteszekelyhidi21,ForsterSzekelyhidi18}, the interface may spontaneously yield to a zone of mixing which grows over time. This is exactly the behaviour that we establish for System \eqref{eq:rho}--\eqref{eq:mu}.
  
\medskip 

\emph{An Entropy Approach to Non-uniqueness:} 
As a novel approach to non-uniqueness, we study the time dissipation of the \textit{relative entropy} \eqref{eq:Phi-intro}.
That is to say, the entropy corresponding to the difference between the entropy of the sum and the sum of the individual entropies.  
This specific choice of entropy is distinguished amongst those which measure the mixing of the two densities because its dissipation does not depend on the pressure gradient. 

\medskip 

As a contrasting example, one could hope to establish a similar result by studying the dissipation of the energy 
\[
\mathcal{G}[\rho,\mu] = \int_{\R} \rho \mu \diff x
\]
which also vanishes if and only if the densities are segregated.
A formal calculation of the dissipation of $\mathcal{G}$ along the flow of System \eqref{eq:rho}--\eqref{eq:mu} shows that this choice of relative entropy produces a term 
\begin{equation}\label{eq:example}
    \int_{\R} \partial_x f'(S) \left(\rho \partial_x \mu + \mu \partial_x \rho\right) \diff x.
\end{equation}
However, due to the inherent formation of inter-species interfaces, $\partial_x \rho, \partial_x \mu$ can not behave better than measures (cf. \cite{2025arXiv250418484M}). 
On the other hand, the pressure gradient $\partial_x f'(S)$ may lack continuity precisely at the interface. 
Consequently, one can not guarantee that the Integral \eqref{eq:example} can be identified or even given meaning in the limiting system.
This means that relative entropies such as $\mathcal{G}$ are unsuitable for our non-uniqueness method.

\medskip 

Having identified a suitable candidate entropy and calculated its dissipation through the viscous approximation, a delicate aspect of our argument is show that the dissipation of the relative entropy has a positive sign over a time interval of positive measure.
In particular, one must show that this sign is preserved in the vanishing viscosity limit. 
By establishing the positivity of the relative entropy, we establish a positive mixing of the two phases after the initial time.  
In contrast, the regularity of segregated solutions is too low to justify this dissipation argument and so, since segregated solutions are not obtained as the limit of a smooth approximating system, we do not obtain a contradiction.

\medskip 

When studying general initial data, a crucial aspect of our analysis is to prove that the interface can neither vanish, nor jump away from the neighbourhood where the two drifts enforce mixing. 
It is essential to retain both properties in the vanishing viscosity limit.
The latter result is shown by exploiting the regularity established for segregated solutions in \cite{santambrogioschulz2026segregated}, however, proving that the interface persists is a more delicate matter, particularly in the regime of slow-diffusion. 

\medskip 

Indeed, if $f(s) \propto s^\alpha, \alpha > 1$, diffusive effects degenerate around areas of low density.
Consequently, the aggregate density may vanish to a vacuum in areas where the drifts encourage the separation of the two phases.
To ensure that this behaviour can not occur when the sign of the drifts is favourable, we prove that the system decouples if a vacuum state occurs around the interface thus yielding the governing equations
\begin{align*}
    \partial_t\rho & = \partial_{xx}(\rho^\alpha) + \partial_x (\rho\partial_x V)\\
      \partial_t\mu & = \partial_{xx}(\mu^\alpha) + \partial_x (\mu\partial_x W).
\end{align*}
Given the decoupled system, we construct sub-solutions inspired by the travelling wave profiles introduced in \cite{ZeldovichKompaneec50} (see also \cite[Chapter X, $\mathsection$ 3]{ZeldovichRaizer67}). 
Moreover, by using the comparison principle proven in \cite{KimZhang18}, the sub-solution construction shows that $\rho$ and $\mu$ must mix as long as the difference of the drifts possesses the correct sign. 
Thus, we contradict the existence of a vacuum state.

\medskip 

\emph{Variable Motility:} As a further generalisation to System \eqref{eq:rholit}--\eqref{eq:Plit} one may consider System \eqref{eq:rho}--\eqref{eq:mu} when the species exhibit a variable motility.
In this regime, weak well-posedness for the one-dimensional problem established in \cite{BertschGurtin87VariableMotility} for segregated solutions. 
The short time existence of classical solutions was then established \cite{DruetHopfJungel23}, followed recently by a global weak existence theory in \cite{ElbarMotility26}. An insightful investigation of the variable motility dynamics was also developed in \cite{LorenziLorzPerthame17}, modelling the system in both one and two dimensions.

\medskip

It is an interesting open problem to consider whether a similar non-uniqueness theory can be proven between the segregated and vanishing viscosity solutions in the variable motility regime.

\subsection{Notations and main results}
We make some assumptions on the different parameters and functions arising in this problem. 

\begin{enumerate}[label=(A\arabic*)]
\item\label{ass:f} The function $f\colon[0,\infty)\to\R$ is strictly convex on $(0,\infty)$ and belongs to $C^2((0,\infty))$.
Since $f$ is strictly convex on $(0,+\infty)$, its derivative is monotone and its inverse is denoted
$$
g=(f')^{-1}\colon \mathrm{Im}(f') \to[0,\infty].
$$
\item\label{ass:VW} The potentials $V,W\in C^3(\R)$ satisfy $\partial_x V, \partial_x W  \in L^\infty(\R)$ and
$$
V(0)=W(0)=0, \quad
\partial_x W(0)-\partial_x V(0)>0, \quad 
\partial_x V-\partial_x W\in C_c^{\infty}(\R).
$$

\item\label{ass:sigma} There exists a number $\sigma_*>0$ such that, writing
$$
A=f'(\sigma_*),
$$
we have
$$
A-V(x) \in \mathrm{Im}(f') \text{ for every } x \leq 0, \quad \quad   A-W(x) \in \mathrm{Im}(f') \text{ for every } x \geq 0.
$$
Further, we require that
$$
g\left(A-V(\cdot)\right)\in L^1(-\infty,0), \quad g\left(A-W(\cdot)\right)\in L^1(0,\infty).
$$
\end{enumerate}

Under these assumptions, we consider the model initial datum
\begin{equation}
\label{eq:initial-data}
\rho_0(x)=g\left(A-V(x)\right)\one_{(-\infty,0]}(x), \quad
\mu_0(x)=g\left(A-W(x)\right)\one_{[0,\infty)}(x).
\end{equation}
By construction,
\[
\rho_0(0)=\mu_0(0)=g(A)=\sigma_*,
\]
and the initial datum is segregated with a single interface at the origin.

\medskip 

Assumption \ref{ass:sigma} is easy to realise but depends on the choice of pressure law. 
\begin{example}\label{example:profile}
Consider first the linear diffusion law (denoted as such as the equation for the sum $\rho+\mu$ is the heat equation)
\[
f'(s) = \log(s).
\]
Then $\mathrm{Im}(f') = \R$. 
Hence, Assumption \ref{ass:sigma} may be satisfied for any $V,W\in C^3(\R)$ satisfying 
\[
\exp(-W(x)), \ \exp(-V(x)) \in L^1(\R).
\]
\end{example}

We now define the notion of weak solutions for System~\eqref{eq:rho}--\eqref{eq:mu}.

\begin{definition}[Weak solution]
\label{def:weak}
Consider non-negative initial data $\rho_0,\mu_0 \in L^1(\R)$.
Then, a pair $\rho,\mu\in L^\infty((0,T);L^1(\R))$
is called a weak solution for~\eqref{eq:rho}--\eqref{eq:mu} on $[0,T]$ if, after setting $S=\rho+\mu$, the functions $\rho\partial_x f'(S)$ and $\mu\partial_x f'(S)$ belong to $L^1_{\mathrm{loc}}((0,T)\times\R)$ and, for every test function $\varphi\in C_c^{\infty}([0,T)\times\R)$, one has
\begin{align*}
-\int_0^T \int_{\R}\rho \partial_t\varphi\diff  x\diff  t+ \int_0^T \int_{\R}\rho \partial_x f'(S) \partial_x\varphi\diff  x\diff  t
+\int_0^T \int_{\R}\rho \partial_x V \partial_x\varphi\diff  x\diff  t
&=\int_{\R}\rho_{0}\,\varphi(0,\cdot)\diff x 
\\
-\int_0^T \int_{\R}\mu \partial_t\varphi\diff  x\diff  t+\int_0^T \int_{\R}\mu \partial_x f'(S) \partial_x\varphi\diff  x\diff  t
+\int_0^T \int_{\R}\mu \partial_x W \partial_x\varphi\diff  x\diff  t
&=\int_{\R}\mu_{0}\,\varphi(0,\cdot)\diff x.
\end{align*}
\end{definition}

The viscous approximation of the system is
\begin{align}
\partial_t\rho - \partial_x(\rho\,\partial_x f'(\rho+\mu))- \partial_x (\rho\,\partial_x V) &= \eps\partial_{xx}\rho,\label{eq:rhoeps}\\
\partial_t\mu  - \partial_x (\mu\,\partial_x f'(\rho+\mu))- \partial_x (\mu\,\partial_x W) &= \eps\partial_{xx}\mu.\label{eq:mueps}
\end{align} 
Given $(\rho_0,\mu_0)$, our initial data for the segregated profile, we let the viscous approximation begin at the initial data $(\rho_{0,n},\mu_{0,n})$ which is chosen to converge weakly to $(\rho_0,\mu_0)$ as $n \to \infty$.
Moreover, the initial data is taken to be smooth and satisfy $\rho+\mu > 0$, so that we know that smooth solution to System \eqref{eq:rhoeps}-\eqref{eq:mueps} exist up to the initial time.
For our purposes, we only need the limit solutions for which the viscosity approximation converges.
\begin{definition}[Vanishing-viscosity solution]
\label{def:vv}
A weak solution $(\rho,\mu)$ of \eqref{eq:rho}--\eqref{eq:mu} is called a vanishing-viscosity solution if there exists a sequence $\eps_n\to 0$ and a sequence of smooth non-negative solutions $(\rho_n,\mu_n)=(\rho_{\eps_n},\mu_{\eps_n})$ decaying at infinity of \eqref{eq:rhoeps}--\eqref{eq:mueps} with initial conditions $(\rho_{0,n},\mu_{0,n})$ such that, writing $S_n=\rho_n+\mu_n$ and $S=\rho+\mu$, the following convergences hold.
\begin{enumerate}[label=(VV\arabic*)]
\item\label{vv1} The initial conditions are such that
$$
(\rho_{0,n},\mu_{0,n}) \rightharpoonup  (\rho_0,\mu_0) \quad \text{weakly in $L^{1}(\R)$}
$$ 
\item\label{vv2} For almost every $t\in(0,T)$,
$$
\rho_n(t,\cdot)\rightharpoonup \rho(t,\cdot)
\quad\text{and}\quad
\mu_n(t,\cdot)\rightharpoonup \mu(t,\cdot)
\quad\text{weakly in }L^1(\R).
$$

\item\label{vv3} There exists $m>\frac 1 2$ such that the functions
$$
q_n=\partial_x\rho_n-\frac{\rho_n}{S_n}\partial_xS_n
= S_n \partial_x \left(\frac{\rho_n}{S_n}\right)
$$
converge weakly in $L^1(0,T; H^{-m}_{loc}(\R))$ to a distribution $q$, and whenever 
$$
\partial_x\rho-\frac{\rho}{S}\partial_x S
$$
makes distributional sense we require that it is the limit $q$.
\end{enumerate}
\end{definition}

\begin{remark}
Definition \ref{def:vv} is exactly the compactness required by our relative entropy method. 
The existence of such solutions is a different theorem, which has been proven in the aforementioned cases. 
Specifically, the works \cite{2025arXiv250418484M,ElbarSantambrogio25,MeszarosParker26,Skrzeczkowski2026crossdiffusion} concern the existence of vanishing viscosity solutions whilst \cite{MeszarosKim2018,santambrogioschulz2026segregated} concern the well-posedness of segregated solutions. Let us mention however that in the slow-diffusion regime ($\alpha>1$), Assumption~\ref{vv3} is not an immediate consequence of the previous works, as it is difficult to pass to the limit in the term
$$
\frac{\rho_n}{S_n}\partial_xS_n.
$$
In the fast-diffusion regime, it is known that $\frac{\partial_x S_n}{S_n}$ converges weakly to $\frac{\partial_x S}{S}$, and $\rho_n$ converges strongly to $\rho$. However, this is not the case anymore in the slow-diffusion regime, where one only obtains a bound on $\partial_x S_n^{\alpha}$ and difficulties related to the zones of vacuum of the sum appear. However the arguments can be localized around the interface, where positivity of the sum holds. This would require to assume the difference of the drifts is compactly supported around the interface as well as an adaptation of~\ref{vv3}, asking instead a weak convergence of the same quantity close to the interface and not necessarily on the full space, that we do not pursue in the core of the paper for simplicity. Instead, we refer to Appendix~\ref{app:porous_medium} for the adaptation in the latter case. 
\end{remark}

We define, with the convention $0\log0=0$,  the \textit{relative entropy}:
\begin{equation}\label{eq:relative_entropy}
\Phi[\rho,\mu] =\int_\R h(\rho(x),\mu(x))\,\diff x, \quad
h(u,v)=(u+v)\log(u+v)-u\log u-v\log v.
\end{equation}

The relative entropy satisfies the following proposition.

\begin{proposition}\label{prop:entropy_properties}
$h$ is nonnegative and concave. 
$$
\Phi[\rho,\mu]=0 \quad \text{if and only if }
\rho\mu=0\ \text{a.e. on }\R.
$$
In particular, $\Phi(t)>0$ implies that $\rho(t,\cdot)$ and $\mu(t,\cdot)$ overlap on a set of positive measure. 
\end{proposition}

\begin{proof}
The map $h$ satisfies $h(0,0) = 0$ and 
$$
\partial_vh(u,v), \partial_u h(u,v) \ge 0 \text{ for } u,v \in [0,\infty)^2 \setminus \{0,0\},
$$
with strict inequality when $u>0$ and $v>0$. Therefore, $h(u,v)\ge 0$ on $[0,\infty)^2 $ and $h(u,v)=0$ iff $uv=0$.
Concavity of $h$ follows by computing its hessian.  
\end{proof}

\subsection{Main results}

We can now state our main theorems. The first one is the construction of an explicit segregated solution of our problem. The second one shows that such a solution cannot be a vanishing viscosity solution.

\begin{thm}[Explicit segregated solution]
\label{thm:stationary}
Assume \ref{ass:f}--\ref{ass:sigma}, and let $(\rho_0,\mu_0)$ be given by \eqref{eq:initial-data}. Define
$$
\bar\rho(t,x)=\rho_0(x),
\quad
\bar\mu(t,x)=\mu_0(x)
\quad\text{for every }(t,x)\in[0,T]\times\R.
$$
Then $(\bar\rho,\bar\mu)$ is a weak solution of \eqref{eq:rho}--\eqref{eq:mu}. It is segregated for every time, has a single interface at $x=0$, and satisfies
$$
\bar\rho(t,0)=\bar\mu(t,0)=\sigma_* \quad\text{for every }t\in[0,T].
$$
\end{thm}
\begin{thm}[Vanishing-viscosity solutions]
\label{thm:vanishing}
Assume \ref{ass:f}--\ref{ass:sigma}. Let $(\rho,\mu)$ be a vanishing-viscosity solution in the sense of Definition \ref{def:vv}. Then, for almost every $t\in(0,T)$,
\begin{equation}\label{eq:entropyineq}
\Phi[\rho(t),\mu(t)]\ge \int_0^t \langle q(s),\partial_x V-\partial_x W\rangle\diff  s.
\end{equation}
In particular, the explicit segregated stationary solution from Theorem \ref{thm:stationary} is not a vanishing-viscosity solution. Here $\langle\cdot, \cdot \rangle$ denotes the distributional evaluation.
\end{thm}

Combining the previous theorems yield a non-uniqueness result. We are aware that an independent construction, by J. Carrillo, A. Jüngel, J. Skrzeczkowski and Y. Yao has been done at the same time to prove non-uniqueness for this system~\cite{jakub_nonunique}.

\begin{corollary}[Non-uniqueness]
\label{cor:nonunique}
Assume \ref{ass:f}--\ref{ass:sigma} and suppose that at least one vanishing-viscosity solution exists for the initial datum \eqref{eq:initial-data}. Then the Cauchy problem \eqref{eq:rho}--\eqref{eq:mu} admits at least two distinct weak solutions on $[0,T]$:
\begin{itemize}
\item the explicit segregated stationary solution $(\bar\rho,\bar\mu)$ from Theorem \ref{thm:stationary}, 
\item any vanishing-viscosity solution $(\rho,\mu)$.
\end{itemize}
\end{corollary}

We can even prove a stronger result, proving non-uniqueness in the larger class of any segregated solutions, not only stationary states. This is based on the adaptation of a result from~\cite{santambrogioschulz2026segregated} that show existence of segregated solutions for general segregated initial profiles.

\medskip

More precisely, we assume:
\begin{enumerate}[label=(B\arabic*)]
    \item \label{ass:positive_2}$\rho_0,\mu_0\in L^{1}(\R)\cap L^\infty(\R)\cap \mathscr{P}_2(\R)$ with $\rho_0(x) = 0$ for $ x > 0 $ and $\mu_0(x) =0 $ for $x< 0$. 
    \item \label{ass:energy_2} $\rho_0,\mu_0$ satisfy $\mathcal{F}[\rho_0,\mu_0] < + \infty$.
    \item \label{ass:interface_2} There exists $r_0>0$, $a_0>0$, $a_1>0$  such that $\rho_0\ge a_0$ a.e. on $(-r_0,0)$ and $\mu_0\ge a_1$ a.e. on $(0,r_0)$.  
    \item \label{ass:pressure_2} let $f(s) = f_\alpha(s)$ for $\alpha \in (\frac{1}{3},+\infty)$ where 
    \[
    f_\alpha(s) = \frac{1}{\alpha-1}s^\alpha \text{ for } \alpha \neq 1 \text{ and } f_1(s) = s\log(s) -s.
    \]
\end{enumerate}

\begin{thm}[Non-uniqueness for general initial data]
\label{thm:nonuniqueness-general}
Assume initial data, pressure law and potentials satisfying ~\ref{ass:f}--\ref{ass:VW} and~\ref{ass:positive_2}--\ref{ass:pressure_2}. 
Then the Cauchy problem \eqref{eq:rho}--\eqref{eq:mu} admits at least two distinct weak solutions on $[0,T]$:
\begin{itemize}
\item the segregated solution constructed in \cite{santambrogioschulz2026segregated},
\item the vanishing viscosity solution constructed in \cite{Skrzeczkowski2026crossdiffusion}.
\end{itemize}
\end{thm}

\begin{remark}
    In the above theorem, we assume a single interface at the origin, however, this assumption is made for the clarity and consistency of exposition. 
    The argument should extend to an arbitrary segregated initial data with discrete number of interfaces.
\end{remark}
\subsection{Contents of the paper}
    In Section 2, we construct explicit stationary and segregated profiles satisfying System \eqref{eq:rho}--\eqref{eq:mu}, thereby proving Theorem \ref{thm:stationary}. 
    In Section 3, we derive Entropy Inequality \eqref{eq:entropyineq} which, in conjunction with the result of the previous section, facilitates the explicit non-uniqueness result: Corollary \ref{cor:nonunique}. 
    In Section 4, we prove that the Entropy dissipation admits a positive sign for general segregated initial data which are well prepared with respect to the sign of the drifts. 
    The paper then concludes, via Inequality \eqref{eq:entropyineq}, with the proof of Theorem \ref{thm:nonuniqueness-general}.

\section{An explicit segregated solution}

Here we prove Theorem~\ref{thm:stationary}, which shows the existence of a segregated solution. The proof is a simple consequence of the fact that the function we choose is a stationary state of the equation. Later, in Theorem~\ref{thm:nonuniqueness-general}, using the result from~\cite{santambrogioschulz2026segregated}, we are able to build a larger class of solutions, that are not necessarily stationary. However, for the non-uniqueness result, Theorem~\ref{thm:stationary} is enough.

\begin{proof}[Proof of Theorem~\ref{thm:stationary}]
We use the following notations in our proof: 
$$
\bar\rho(x)=g\left(A-V(x)\right)\one_{(-\infty,0]}(x),
\quad \bar\mu(x)=g\left(A-W(x)\right)\one_{[0,\infty)}(x),
$$
and define $(\bar\rho(t,x),\bar\mu(t,x))=(\bar\rho(x),\bar\mu(x))$ for every $t\in[0,T]$. 
By construction and assumptions~\ref{ass:sigma}, $\bar\rho$ and $\bar\mu$ are non-negative. 
Moreover, we have
$$
\supp \bar{\rho}\subset (-\infty,0], \quad \supp \bar{\mu}\subset[0,+\infty).
$$
We define $\bar{S} = \bar{\rho} + \bar{\mu}$. Passing to the limits as $x\to0^-$ and $x\to0^+$, and using $V(0)=W(0)=0$, we obtain that $\bar{S}$ can be extended continuously at 0, with $\bar{S}(0)=g(A)=\sigma_*$. Then, on the left-half line $\bar{S}=\bar{\rho}$ and:
$$
f'(\bar S(x))+V(x)=f'(g(A-V(x)))+V(x)=A-V(x)+V(x)=A
\quad\text{for }x<0.
$$
Similarly,
$$
f'(\bar S(x)) + W(x) = A \quad \text{for }x>0.
$$

We can take the derivative $\partial_x$ on each side of the line, and then multiplying by $\bar{\rho}$ and $\bar{\mu}$, using the range of their support we obtain 
$$
\bar\rho \partial_x f'(\bar S)+\bar\rho \partial_x V=0,
\quad \bar\mu \partial_x f'(\bar S)+\bar\mu \partial_x W=0,\quad  \text{for a.e. }x\in\R.
$$
It is then clear that $(\bar{\rho}, \bar{\mu})$ is a stationary solution (and therefore a weak solution) of~\eqref{eq:rho}--\eqref{eq:mu}.
\end{proof}

\section{Vanishing viscosity}

We claim that the solution obtained in Theorem~\ref{thm:stationary} cannot be a vanishing viscosity solution in the sense of~\ref{def:vv}. This uses a bound from below on the \textit{relative entropy} introduced in~\eqref{eq:relative_entropy}. In order to prove Theorem~\ref{thm:vanishing}, we first compute the dissipation of the relative entropy at the viscous level, where the computations are exact since the solutions are classical. The following proposition is an immediate computation, using~\eqref{eq:rhoeps}--\eqref{eq:mueps}.

\begin{proposition}[Entropy identity at the viscous level]\label{prop:entropy-viscous}
Let $\eps>0$, and let $(\rho_\eps, \mu_\eps)$ be a smooth positive solution of~\eqref{eq:rhoeps}--\eqref{eq:mueps}. Let 
$$
S_\eps= \rho_\eps + \mu_\eps \quad q_\eps = \partial_x\rho_\eps-\frac{\rho_\eps}{S_\eps}\partial_xS_\eps
=S_\eps \partial_x \left(\frac{\rho_\eps}{S_\eps}\right).
$$    
Then for every $t\in(0,T)$: 
\begin{equation*}
\frac{\diff}{\diff t}\Phi[\rho_\eps(t),\mu_\eps(t)]
=
\int_{\R}(\partial_x V-\partial_x W)q_\eps\diff  x
+\eps\int_{\R}\left(\frac{|\partial_x\rho_\eps|^2}{\rho_\eps}
+\frac{|\partial_x\mu_\eps|^2}{\mu_\eps}
-\frac{|\partial_xS_\eps|^2}{S_\eps}\right)\diff  x.
\end{equation*}
In particular, since
$$
\frac{|\partial_x\rho_\eps|^2}{\rho_\eps}
+\frac{|\partial_x\mu_\eps|^2}{\mu_\eps}
-\frac{|\partial_xS_\eps|^2}{S_\eps} = \frac{\left(\mu_\eps\partial_x \rho_\eps - \rho_\eps\partial_x\mu_\eps\right)^2}{\rho_\eps\mu_\eps S_\eps}
$$
we obtain
\begin{equation*}
\frac{\mathrm{d}}{\mathrm{d}t}\Phi[\rho_\eps(t),\mu_\eps(t)] 
\ge \int_{\R}(\partial_x V-\partial_x W)q_\eps\diff  x .
\end{equation*}
\end{proposition}

\begin{proposition}[Limit entropy]
\label{prop:limit-entropy}
Let $(\rho,\mu)$ be a vanishing-viscosity solution in the sense of Definition \ref{def:vv}, and let $q$ be the distribution obtained in \ref{vv3}. Then for almost every $t\in(0,T)$ one has
$$
\Phi[\rho(t),\mu(t)]\ge \int_0^t \langle q(s),\partial_x V-\partial_x W\rangle\diff  s.
$$
\end{proposition}

\begin{proof}[Proof of Proposition~\ref{prop:limit-entropy}]
Let $(\rho_n,\mu_n)$ be the viscous approximating sequence from Definition \ref{def:vv}, and let $S_n=\rho_n+\mu_n$. By Proposition \ref{prop:entropy-viscous}, for every $n$ and every $t\in(0,T)$ we have

\begin{equation*}
\Phi[\rho_n(t),\mu_n(t)]
\ge \Phi[\rho_n(0),\mu_n(0)]
+\int_0^t \int_{\R}(\partial_x V-\partial_x W)q_n\diff  x\diff  s.
\end{equation*}
By Proposition \ref{prop:entropy_properties}, the energy $\Phi$ is positive, and so
\begin{equation*}
\Phi[\rho_n(t),\mu_n(t)]
\ge \int_0^t \int_{\R}(\partial_x V-\partial_x W)q_n\diff  x\diff  s.
\end{equation*}
We now pass to the limit. Since $\Phi$ is concave, it is upper semi-continuous for the weak $L^1$ convergence, and we obtain
$$
\Phi[\rho(t),\mu(t)]\ge \limsup_{n\to\infty}\Phi[\rho_n(t),\mu_n(t)].
$$

Since $\partial_xV$ and $\partial_xW$ are smooth and their difference is compactly supported, we can then pass to the limit, using $q_n\to q$ weakly in $L^1(0,T; H^{-m}_{loc}(\R))$. This concludes the proof.
\end{proof}

In order to complete the proof of Theorem~\ref{thm:vanishing}, it remains to show that the segregated solution constructed in Theorem~\ref{thm:stationary} cannot be a viscosity solution. This is a consequence of the following lemma.

\begin{lemma}
\label{lem:dirac_interface}
Let $I\subset\R$ be an interval, let $\eta\in I$, and let $S\in W^{1,1}(I)\cap C(I)$ be a continuous non-negative function. Define
$$
\rho(x)=S(x)\one_{(-\infty,\eta)}(x), \quad \mu(x)=S(x)\one_{(\eta,\infty)}(x) \quad (x\in I).
$$
Then, in the sense of distributions on $I$,
$$
\partial_x\rho-\frac{\rho}{S}\partial_xS=-S(\eta)\delta_{\eta},
\quad \partial_x\mu-\frac{\mu}{S}\partial_xS=S(\eta)\delta_{\eta}.
$$
\end{lemma}

With this lemma at hand, we can now prove Theorem~\ref{thm:vanishing}.

\begin{proof}[Proof of Theorem~\ref{thm:vanishing}]
The first part of Theorem~\ref{thm:vanishing} is a consequence of Proposition~\ref{prop:limit-entropy}. Now assume by contradiction that $(\bar\rho,\bar\mu)$ is a vanishing-viscosity solution. Let $\bar S=\bar\rho+\bar\mu$. Since the solution is stationary, segregated, has a single interface at $x=0$, and $\bar S$ is continuous at the interface with value $\bar S(0)=\sigma_*$, Lemma~\ref{lem:dirac_interface} gives
$$
q=\partial_x\bar\rho-\frac{\bar\rho}{\bar S}\partial_x\bar S=-\sigma_*\delta_0.
$$ 

Since the solution is stationary and segregated, Proposition \ref{prop:entropy_properties} yields
$$
\Phi[\bar\rho(t),\bar\mu(t)]=0 \quad\text{for every }t\in[0,T].
$$
On the other hand,
\begin{align*}
\int_0^t \langle q(s), \partial_x V-\partial_x W\rangle\diff  s
&=\int_0^t \sigma_*\left(\partial_x W(0)-\partial_x V(0)\right)\diff  s
\\
&=t \sigma_*\left(\partial_x W(0)-\partial_x V(0)\right).
\end{align*}
which is a contradiction since, by assumption, $\sigma_{*}>0$ and $\partial_x W(0)-\partial_x V(0)>0$.
\end{proof}

\section{Non-uniqueness for general initial data}

In this section, we prove Theorem~\ref{thm:nonuniqueness-general}. 
As with the strategy for proving Theorem~\ref{thm:stationary}, Theorem~\ref{thm:nonuniqueness-general} is shown by proving that a segregated vanishing viscosity solution of System \eqref{eq:rho}-\eqref{eq:mu} cannot exist. 
Indeed, any vanishing viscosity solution must satisfy Inequality \eqref{eq:entropyineq}.
Thus, given suitable potentials and initial datum, the sign of the relative entropy contradicts the segregation property. 
In extending the approach to general initial data the novel considerations which concern the interface are two-fold. 

\medskip 

\begin{enumerate}
    \item \underline{Localisation:} Firstly, by Assumptions \ref{ass:positive_2}--\ref{ass:interface_2}, the datum $(\rho_0,\mu_0)$ possesses a single interface located at $0$ whilst, by Assumption \ref{ass:VW}, the drift satisfies $\partial_x(W(0)-V(0)) > 0$. 
    Since $W,V$ are smooth, one may also assert that $\partial_x(W-V)$ is strictly positive in a small neighbourhood of $0$.
    Consequently, a first step is to show that, if the interface exists, then it belongs to this neighbourhood for a small time interval $[0,\tau]$. 

    \item \underline{Persistence:} Secondly, since $\rho,\mu$ may evolve over time, it is feasible that the two species separate from one another, causing the interface to vanish. 
    In this scenario, the relative entropy \eqref{eq:relative_entropy} is zero and the non-uniqueness argument fails. 
    Thus, a second step in asserting the positivity of the relative entropy is to show that the densities cannot separate. 
    Specifically, we prove that the aggregate density preserves positivity in a small neighbourhood of the origin and for a small time interval $[0,\tau]$.
\end{enumerate}

\subsection{Structure for segregated solutions}
Before proving the persistence and localisation of the interface, we recall, from \cite{santambrogioschulz2026segregated}, the structure of the therewithin established segregated solutions under our standing assumptions \ref{ass:f}--\ref{ass:VW} and \ref{ass:positive_2}--\ref{ass:pressure_2}. 
These assumptions will be tacitly assumed throughout the remainder of this manuscript unless stated.
Moreover, in referring to a segregated solution to \eqref{eq:rho}-\eqref{eq:mu} we will always mean the solution which is established in \cite[Theorem 3.5]{santambrogioschulz2026segregated}.
\begin{proposition}\label{propn:segregationproperties}
Let $(\rho,\mu)$ denote a segregated solution to \eqref{eq:rho}-\eqref{eq:mu}.
Then, there exists $u \colon [0,T]\times (0,2)\to \R $ which satisfies all of the following properties:
\begin{enumerate}
\item The function $u_0 \coloneqq u(0, \cdot )$ satisfies $u_0(1) = 0$ and is continuous in a neighbourhood of $1$.
\label{Prop4}
\item The map $ (0,2) \ni y \mapsto u_t(y) $ is strictly increasing for almost every $ t \in [0,T]$. 
\label{Prop2}
\item The function $u$ admits the regularity $ H^1(0,T;L^2[0,2]) \cap L^\infty(0,T; BV_{loc}(0,2))$.
\label{Prop3}
\item The densities $\rho,\mu$ satisfy the equality \[
\rho_t(x) = S_t(x)\one_{(-\infty,u_t(1)]}(x) , \quad  \mu(x) = S_t(x)\one_{(u_t(1),+\infty)}(x)
\]
for almost every $x \in \R$, for almost every $t \in [0,T]$.
\label{Prop1}
\end{enumerate}
\end{proposition}
\begin{proof}
We construct an suitable initial datum $u_0$ from the prescribed pair $(\rho_0,\mu_0)$. 
To do so, we first let $u_0$ denote a representative of the pseudo-inverse of the function $F_0\colon \R \to [0,2]$ defined by
\[
F_0(x) = \int_{-\infty}^x \rho_0(x') + \mu_0(x') \diff x'
\]
such that $u_0 \in L^2([0,2])$ - (this choice is possible since $\rho_0, \mu_0 \in \mathscr{P}_2(\R)$). 
By assumption \ref{ass:interface_2}, it follows that there exists a neighbourhood of the origin for which $F_0$ is continuous and strictly monotone. 
Since $F_0(0)=1$, the image of this neighbourhood must be an open interval containing $1$. 
Moreover, on this interval, $u_0$ is continuous and coincides with the classical inverse of $F_0$, satisfying $u_0(1) = u_0(F_0(0)) = 0$. Thus, Point \ref{Prop4} is satisfied.

\medskip 

Choosing $u_0$ as the initial datum for the problem \cite[Proposition 6.1, Equation 6.1]{santambrogioschulz2026segregated}, it follows from \cite[Theorem 3.6]{santambrogioschulz2026segregated} that there exists a solution $u \colon [0,T]\times (0,2) \to \R$, which we will use to define a segregated solution. 
In particular, this solution satisfies the initial data in the sense $\lim_{t\to 0} \|u_t - u_0\|_{L^2(\R)} = 0$.

\medskip 

    We now refer to the construction of solutions given in the proof of \cite[Theorem 3.5; Section 7]{santambrogioschulz2026segregated}. 
    Indeed, from this construction, it follows that $u$ satisfies Points \ref{Prop2} and \ref{Prop3} and that there exists a family of sets $(G_t)_{t \in [0,T]}$, which are of full measure in $[0,2]$, such that the segregated solution of \eqref{eq:rho}--\eqref{eq:mu} may be defined, for almost every $t \in [0,T]$, by the following construction.
    
    \medskip 

    We first recall the function $F_t$ defined as the inverse of $u$. 
    In particular this function is well-defined on the image of $u_t$ because $u_t$ is strictly increasing.
    \begin{equation}\label{eq:F_t}
           F_t \colon u_t((0,2)) \to (0,2), \quad F_t(x) \coloneqq u_t^{-1}(x)\\
    \end{equation}
    Now, recall from \cite[Proof of Theorem 3.5]{santambrogioschulz2026segregated} that $F_t$ is differentiable almost everywhere on the set $u_t(G_t)$ and that $u_t(G_t)$ is open in $\R$. 
    Thus, the segregated solution may be defined as follows.
    \begin{alignat}{2}
        S_t(x) & = \partial_x F_t(x)\one_{u(G_t)}(x)  \label{eq:St} \\
       \rho_t(x) & = S_t(x) \one_{(-\infty, u_t(1)]}(x),   \quad  \mu_t(x) = S_t(x) \one_{(u_t(1),+\infty)}(x), \label{eq:rho_tmu_t}
    \end{alignat}
\end{proof}

Having established some key properties for the segregated solutions, we prove that any interface between the two densities can not spontaneously transport a large distance away from the origin in a short time.

\begin{proposition}\label{propn:localisation}
Let $(\rho,\mu)$ denote a segregated solution to \eqref{eq:rho}-\eqref{eq:mu} and let $U$ denote an open neighbourhood of the origin.
Then, there exists $\tau,\delta > 0$ such that $u_t(y) \in U$ for every $y \in [1-\delta, 1+\delta]$, for almost every $t \in [0,\tau]$. 
 \end{proposition}
\begin{proof}
Point \ref{Prop4} of Proposition \ref{propn:segregationproperties} asserts that $u_0(1) = 0$ and that $1$ is a point of continuity for $u_0$.
Thus, for an open neighbourhood of the origin, denoted $U$, there exist $y_1, y_2 \in (0,2)$ satisfying $y_1 < 1 < y_2$ and for which $u_0(y_1),u_0(y_2) \in U$. 
It also follows from Point \ref{Prop3} that $\partial_t u \in L^2(0,T;L^2_{loc}(0,2))$. 
Hence, consequent to the one dimensional Sobolev embedding, the map $[0,T]\ni t \mapsto u_y(t)$ is $\frac{1}{2}$-H\"older continuous for almost every $y \in (0,2)$.
In particular, $y_1,y_2$ may be chosen as such points of H\"older continuity.

\medskip 

Now, since the set $U$ is open, the continuity of the map $t\mapsto u_{y_i}(t)$ yields the existence of $\tau > 0 $ such that $u_{y_i}(t) \in U $ for every $t \in [0,\tau]$ and each $i \in \{1,2\}$.
Recall from Point \ref{Prop2} that $y \mapsto u_t(y)$ is strictly increasing for almost every $t \in [0,T]$. 
Hence, by defining $\delta = \min_{i \in \{1,2\}} |1- y_i|$, it follows that $u_t(y) \in U $ for every $y \in [1-\delta, 1+\delta]$, for almost every $t \in [0,\tau]$.
\end{proof}

To make sense of the distributional derivative of a single species when evaluated at the interface, we further establish the continuity of the aggregate density by means of the one dimensional Sobolev embedding and the regularity established in \cite{santambrogioschulz2026segregated}.
\begin{proposition}\label{propn:sumcontinuity}
     Let $(\rho,\mu)$ denote a segregated solution to \eqref{eq:rho}-\eqref{eq:mu}.
     Then $S_t$ admits a continuous representative for almost every $t \in [0,T]$.
\end{proposition}
\begin{proof}
    By \cite[Theorem 3.5]{santambrogioschulz2026segregated} and assumptions \ref{ass:positive_2}--\ref{ass:energy_2}, the segregated solution satisfies 
    \[\sqrt{S} \partial_x f'(S) \in L^2_{loc}((0,T];L^2(\R)).
    \]Moreover $\sqrt{S}$ is bounded in $L^\infty((0,T);L^2(\R))$ by conservation of mass and so, by Cauchy-Schwarz Inequality, it follows that 
    \[S \partial_x f'(S)=\sqrt{S}\sqrt{S}\partial_x f'(S) \in  L^1_{loc}((0,T];L^1(\R)).
    \]
    By Assumption \ref{ass:pressure_2} the pressure satisfies $S \partial_x f'(S) \propto \partial_x(S^ \alpha)$. 
    Thus, from the inclusion $W^{1,1}_{loc}(\R) \subset AC_{loc}(\R)$ and the continuity of the inverse map $s \mapsto s^\frac{1}{\alpha}$ it follows that $S_t$ admits a continuous representative for almost every $t \in [0,T]$.
\end{proof}

\begin{remark}
In the rest of the section, the notation $S_t$ always refers to this continuous representative. 
\end{remark}

\subsection{Interface persistence}
Under our standing Assumption \ref{ass:VW}, the two drifts point towards one another so as to encourage the mixing of the two phases, whilst, due to \ref{ass:interface_2}, the initial aggregate density is positive around the origin. 
Under these hypotheses, it is proven by contradiction that the interface can not vanish by means of the aggregate density obtaining a state of vacuum. 

\medskip 

Indeed, in supposing that a vacuum state occurs at the interface of the two densities, we show that each of the two species satisfy a non-linear drift-diffusion equation and their evolution becomes decoupled. 
Such parabolic equations satisfy comparison principles and Harnack inequalities, thus, it is shown that the independently evolving densities will overlap, contradicting the existence of a vacuum state. 

\begin{proposition}\label{propn:separate}
    Let $(\rho,\mu)$ denote a segregated solution to \eqref{eq:rho}-\eqref{eq:mu} and let $\tau > 0$. 
    Further, suppose that the equality
    \[
     S(t, u_t(1)) = 0
    \]
    is satisfied for almost every $t \in [0,\tau]$.
    Then, $(\rho,\mu)$ coincides with a weak solution to the de-coupled system
    \begin{equation}\label{eq:decoupled}
    \begin{cases}
        \partial_t \rho &= \partial_{xx}(\rho^\alpha) + \partial_x(\rho\partial_x V), \\
          \partial_t \mu &= \partial_{xx} (\mu^\alpha) + \partial_x(\mu\partial_x W),    
    \end{cases}
    \quad \text{ on } [0,\tau) \times \R.
    \end{equation}
\end{proposition}
\begin{proof}
First, recall from Proposition \ref{propn:sumcontinuity} that $S_t$ admits a continuous representative for almost every $t \in [0,T]$. 
Thus, the evaluation $S_t(u_t(1))$ is meaningful.
Secondly, recall that $(\rho,\mu)$ satisfy Points \ref{Prop4}- \ref{Prop1} established in Proposition \ref{propn:segregationproperties}. 
Consequently, for any test function $\varphi \in C_c^\infty ((0,\tau)\times \R)$, the following equality holds.
\begin{align*}
    \int_0^{\tau}\int_{\R} & \varphi \rho \partial_x f'(S)   \diff x \diff t = \int_0^{\tau}\int_{-\infty}^{u_t(1)}\varphi S\partial_x f'(S) \diff x \diff t =  \int_0^{\tau}\int_{-\infty}^{u_t(1)} \varphi \partial_x S^\alpha  \diff x \diff t\\
    & = \int_0^{\tau}\varphi(t,u_t(1)) S^\alpha(t,u_t(1)) \diff t- \int_0^{\tau}\int_{-\infty}^{u_t(1)} \partial_x \varphi S^\alpha   \diff x \diff t\\
    & = - \int_0^{\tau}\int_{\R} \partial_x \varphi \one_{(-\infty,u_t(1)]} S^\alpha  \diff x \diff t= - \int_0^{\tau}\int_{\R} \partial_x \varphi \rho^\alpha \diff x \diff t.
\end{align*}
From the above system of equalities it follows that the equality 
\begin{equation}\label{eq:dxrho}
\rho\partial_x f'(S) = \partial_x \rho^\alpha
\end{equation}
holds in the sense of distributions on $(0,\tau) \times \R$.
Similarly, it follows that
\begin{equation}\label{eq:dxmu}
\mu\partial_x f'(S) = \partial_x \mu^\alpha.
\end{equation}
To conclude, substitute equalities \eqref{eq:dxrho} and \eqref{eq:dxmu} within the weak formulation for System  \eqref{eq:rho}--\eqref{eq:mu} taken over the interval $(0,\tau)\times \R$.
Consequently, one recovers that $(\rho,\mu)$ satisfies the weak formulation of System \eqref{eq:decoupled}.
\end{proof}
Having proven that $(\rho,\mu)$ decouple in the absence of an interface, it is now proven that the two decoupled solutions must instantaneously mix with one another. 
The argument divides into three cases with a more natural technique in each of the diffusive regimes.
\begin{proposition}[Linear Diffusion]\label{propn:linear}
    Let $(\rho,\mu)$ denote a solution of System \eqref{eq:decoupled} with $\alpha =1$. 
    Then $\rho,\mu > 0$ on $(0,T]\times \R$.
\end{proposition}
\begin{proof}
    Let $[a,b]$ be a bounded interval and fix a further bounded interval $[\bar{a},\bar{b}]$ such that $\rho_0$ is not identically 0 over $(\bar{a},\bar{b})$ and such that $[a,b] \subseteq [\bar{a},\bar{b}]$.
    Such a choice is always possible since $\rho_0$ is of mass one.
    Then, since $V,W \in C^3(\R)$ it follows from \cite[Theorem 5.2]{Ladyzhenskaya68} that $\rho, \mu \in C^{1+\frac{\alpha}{2},2+\alpha}_{loc}((0,T]\times \R)$.
    \medskip

    The function $e^{\lambda t}\rho$ coincides with the function $u$, which denotes the unique classical solution of the following boundary problem (cf. \cite[Theorem 5.3]{Ladyzhenskaya68})
    \begin{equation}
    \begin{cases}
      \partial_t u = \partial_{xx} u + \partial_x V\partial_x u + (\lambda + \partial_{xx}V)u  \text{ in }  (0,T)\times (\bar{a},\bar{b}) \\
      u(0,\cdot) = \rho_0(\cdot) \text{ on } (\bar{a},\bar{b}) \\ 
      u(t,\cdot) = e^{\lambda t} \rho(t,\cdot) \text{ on } (0,T) \times \{\bar{a},\bar{b}\}.
      \end{cases}
      \end{equation}
    Since $V \in C^3(\R)$, $\lambda$ may be chosen, for each $(\bar{a},\bar{b})$, such that $\lambda  + \partial_{xx}V \geqslant 0$. 
    Then, since $u \geqslant 0$ on the parabolic boundary and since $u$ is not constant, it follows from the strong minimum principle that $u > 0 $ on $(0,T] \times (a,b)$ (cf. \cite[Theorem 12, Chapter 7]{EvansPDE}).
    The interval $(a,b)$ was chosen arbitrarily, so it follows $\rho = e^{-\lambda t} u > 0$ on $(0,T] \times \R$. 
    The proof holds analogously for $\mu$ and $W$ to establish the positivity of the second density.
\end{proof}
\begin{proposition}[Fast-Diffusion]\label{propn:fast}
    Let $(\rho,\mu)$ denote a solution of System \eqref{eq:decoupled} with $\alpha \in (\frac{1}{3},1)$.
    Further, suppose that $\rho,\mu \in L^\infty([0,T]\times\R)$.
    Then $\rho,\mu > 0$ on $(0,T]\times \R$.
\end{proposition}
\begin{proof} 
It appears that a comparison principle for singular drift-diffusion equations is missing within the current literature. 
Consequently, the result is proven via the Harnack Inequality for fast diffusion: \cite[Corollary 16.1]{DiBenedettoGianazzaVespri12}. 
The proof is presented for the density $\rho$ whilst a completely analogous argument holds for $\mu$.

\medskip 

To proceed, it is verified that $\rho$ satisfies the structural condition specified in \cite[Inequality 5.2, Page 33]{DiBenedettoGianazzaVespri12}. 
This condition states that there must exist $C_0,C_1, C > 0$ for which $\rho$ satisfies the following two inequalities. 
\begin{align}
    & (\partial_x \rho^\alpha + \rho \partial_x V) \cdot \partial_x \rho \geq C_0 \alpha |\rho|^{\alpha-1} |\partial_x \rho|^2 - C^2 |\rho|^{\alpha+1} \label{eq:structure1}\\
    & |\partial_x \rho^\alpha + \rho \partial_x V| \leq C_1 \alpha |\rho|^{\alpha-1}|\partial_x\rho|+ C|\rho|^\alpha. \label{eq:structure2}
\end{align}
To verify Inequality \eqref{eq:structure1}, we write 
\begin{align*}
    (\partial_x \rho^\alpha + \rho \partial_x V) \cdot \partial_x \rho = \alpha \rho^{\alpha-1} |\partial_x \rho|^2 + (\rho^{\frac{\alpha-1}{2}}\partial_x\rho) \cdot (\partial_x V \rho^\frac{3-\alpha}{2}).
\end{align*}
Via Young's Inequality, it follows that 
\begin{align*}
    (\partial_x \rho^\alpha + \rho \partial_x V) \cdot \partial_x \rho \geq \frac{\alpha}{2} |\rho|^{\alpha-1} |\partial_x \rho|^2 - \frac{1}{2\alpha} |\partial_x V|^2 \rho^{3-\alpha}.
\end{align*}
Letting $c = \|\rho^{1-\alpha}\partial_x V \|_{L^\infty([0,T]\times \R)}$, the above inequality then yields 
\begin{align*}
    (\partial_x \rho^\alpha + \rho \partial_x V) \cdot \partial_x \rho \geq \frac{\alpha}{2} |\rho|^{\alpha-1} |\partial_x \rho|^2 - \frac{c^2}{2\alpha} |\rho|^{\alpha+1}
\end{align*}
which verifies Inequality \eqref{eq:structure1}.
In particular, $c$ is finite since $\partial_x V \in L^\infty(\R)$ and $\rho \in L^\infty([0,T]\times \R)$.
To verify Inequality \eqref{eq:structure2}, it is sufficient to recognise that 
\begin{align*}
    |\partial_x\rho^\alpha + \rho \partial_x V| \leq \alpha |\rho|^{\alpha-1}|\partial_x \rho| + c|\rho|^\alpha.
\end{align*}
Having checked the sufficient structural conditions, it follows from \cite[Corollary 16.1]{DiBenedettoGianazzaVespri12} that there exists $\Lambda > 0$, which depends only on $c$, and $\gamma = \gamma (\alpha) > 0$, such that the inequality 
\begin{equation}\label{eq:SingularHarnack}
 \gamma^{-1}\sup_{x \in [x_0-\sigma, x_0 + \sigma ]}\rho(t_0,x) \leq \rho(t_0,x_0) \leq \gamma \inf_{x \in [x_0-\sigma, x_0 + \sigma]}\rho(t_0,x)   
\end{equation}
is satisfied for any $\sigma \in (0,\Lambda)$ and any $(t_0,x_0) \in (0,T]\times \R$ for which $t_0$ satisfies
\begin{equation}\label{eq:sigmabound}
   t_0 \geq 64 \sigma^2\|\rho\|_{L^\infty([0,T]\times\R)}^{1-\alpha}. 
\end{equation}
Notice that, in contrast to the classical Harnack Inequality in the uniformly parabolic regime, this inequality is elliptic in the sense that \eqref{eq:SingularHarnack} only considers data at time $t_0$.

\medskip 

Now, fix $t_0 \in (0,T], x_1 \in \R$ and let $x_0 \in \R$ be such that $\rho(t_0,x_0) > 0$. 
Further, fix $\sigma > 0$ such that Inequality \eqref{eq:sigmabound} is satisfied and choose $k \in \N $ be such that $x_1 \in [x_0 - k\sigma, x_0+k \sigma ]$.
To conclude that $\rho(t_0,x_1) > 0$, Inequality \eqref{eq:SingularHarnack} is iterated in the following fashion.
\begin{align*}
    \rho(t_0,x_0) & \leq \gamma \inf_{x \in [x_0-\sigma, x_0 + \sigma]}\rho(t_0,x)  \leq   \gamma^2 \inf_{x \in [x_0-2\sigma, x_0 + 2\sigma]}\rho(t_0,x)\\
    \dots & \leq \gamma^k \inf_{x \in [x_0-k\sigma, x_0 + k\sigma]}\rho(t_0,x) \leq \rho(t_0,x_1).
\end{align*}
Since the point $(t_0,x_1) \in (0,T]\times \R$ was chosen arbitrarily, we conclude the claim.
\end{proof}

\begin{remark}
    Since the equation for the sum can also be written as a parabolic equation equipped with an $L^\infty_{t,x}$ drift, we could have applied similar arguments to prove that the sum can not produce a vacuum state without the decoupling argument. 
    On the other hand, the following argument for the slow diffusion necessitates the decoupling of the system.
\end{remark}
As a juxtaposition to the linear and fast-diffusive regimes, the slow-diffusive pressure law does not exhibit an infinite speed of propagation as the density approaches vacuum. 
Moreover, it is not guaranteed that the interface will persist in the presence of drift terms which encourage the separation of the two densities. 
Consequently, we utilise the specific structure of the drift terms around the origin to construct pressure sub-solutions.
Due to the comparison principle, these sub-solutions ensure a mixing zone around the origin.
\begin{definition}[Pressure Sub-solution]
    Given $\rho, \mu$ satisfying \eqref{eq:decoupled}, we define the pressure variables 
    \[
    P_\rho = \frac{\alpha}{\alpha-1}\rho^{\alpha-1}, \quad P_\mu = \frac{\alpha}{\alpha-1}\mu^{\alpha-1}
    \] 
    If functions $\pi_\rho,\pi_\mu \colon [0,T] \times \R \to \R$ satisfy the following inequalities in the sense of distributions:
\begin{align} 
& \partial_t \pi_\rho - (\alpha-1)\pi_\rho \partial_{xx}^2(\pi_\rho+V) - |\partial_x \pi_\rho|^2 - \partial_x \pi_\rho \partial_x V \leq 0,\label{eq:subsolutionrho}\\
& \partial_t \pi_\mu - (\alpha-1)\pi_\mu \partial_{xx}^2(\pi_\mu+W) - |\partial_x \pi_\mu|^2 - \partial_x \pi_\mu \partial_x W \leq 0,\label{eq:subsolutionmu}
\end{align}
then $\pi_\rho,\pi_\mu$ define pressure sub-solutions of System \eqref{eq:decoupled}.
\end{definition}
\begin{remark}\label{rmk:comparison}
Due to the comparison principle: \cite[Theorem 3.4]{KimZhang18} any pressure-sub solution of System \eqref{eq:decoupled} which satisfies $\pi_{\rho,0},\pi_{\mu,0} \leq P_{\rho,0}, P_{\mu,0}$ on $\R$ must further satisfy $\pi_{\rho},\pi_{\mu} \leq P_{\rho}, P_{\mu}$ on $[0,T] \times \R$.
\end{remark}

\begin{proposition}[Slow Diffusion]\label{propn:slow}
    Let $(\rho,\mu)$ denote a solution of System \eqref{eq:decoupled} with $\alpha > 1$, with initial data satisfying~\ref{ass:positive_2}--\ref{ass:interface_2}, and assume $\partial_x W(0)-\partial_xV(0)>0$.
    Then there exists $\hat{\tau} > 0$ and a family of non-empty open intervals $(I_t)_{t\in (0,\hat{\tau})}$ for which $\rho_t,\mu_t > 0$ on $I_t\subset \R$ for each ${t\in (0,\hat{\tau})} $.
\end{proposition}
\begin{proof}
The construction of pressure sub-solutions is made in three steps. 
First, we construct a sub-solution for the mixing zone, following a travelling wave ansatz. 
Then, so that we have a sub-solution on the whole domain, the sub-solution is extended by a quadratic profile. 
Lastly, we show that the sub-solutions must mix instantaneously. Then, by means of the comparison principle, it follows that the $(\rho,\mu)$ must also mix instantaneously. For reading clarity, we provide below a plot of the construction of the subsolution that is achieved in the proof.

\begin{figure}[H]
\centering
\begin{tikzpicture}
\def\cmu{-0.9}
\def\crho{1}
\def\eps{2}

\draw[->] (-5,0) -- (5,0) node[below] {$x$};
\draw[->] (0,0) -- (0,3) node[left] {$P$};

\draw[blue,thick]   (-4.6,0) parabola bend (-3.2,1.9) (-\eps,0.95) -- (\crho,0);
\draw[orange,thick] (4.6,0)  parabola bend (3.2,1.9)  (\eps,0.95)  -- (\cmu,0);

\draw[dashed] (-\eps,0) -- (-\eps,1);
\draw[dashed] (\eps,0) -- (\eps,1);
\node[below] at (-\eps,0) {$-\varepsilon$};
\node[below] at (\eps,0) {$\varepsilon$};
\node[below] at (0,0) {$0$};
\node[blue,below] at (\crho,0) {$c_\rho t$};
\node[orange,below] at (\cmu,0) {$c_{\mu} t$};
\node[blue] at (-3.55,2.2) {$\tilde{\pi}_\rho$};
\node[orange] at (3.55,2.2) {$\tilde{\pi}_\mu$};

\draw[decorate,decoration={brace,mirror,amplitude=4pt}](\cmu,-0.55) -- (\crho,-0.55) node[midway, below=3pt]{$I_t$};

\node[blue] at (-1, 1.1) {$\pi_\rho$};
\node[orange] at (1, 1.1) {$\pi_\mu$};
\end{tikzpicture}
\caption{Pressure subsolutions in the slow diffusion case}
\end{figure}
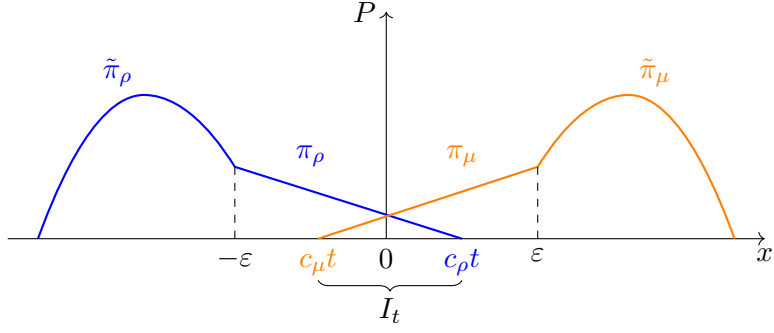

\underline{Part i. Travelling Wave:}
Let $\e \in (0,1)$ and consider $\pi_\rho,\pi_\mu$ satisfying the following ansatz.
\begin{equation}\label{eq:travellingwave}
\begin{split}
\pi_\rho(t,x) & = \lambda(t) (c_\rho t-x)_{+} \text{ for } (t,x) \in [0,\tau] \times (-\e,+\infty), \\
\pi_\mu(t,x) & = \lambda(t) (x- c_\mu t)_{+} \text{ for } (t,x) \in [0,\tau] \times (-\infty,\e).
\end{split}
\end{equation}
where $\lambda\colon [0,\tau] \to [0,+\infty)$ and $c_\rho-c_\mu > 0 $ are to be determined.

\medskip 

Due to the travelling wave construction, the following equalities are satisfied on $[0,\tau] \times (-\e,+\infty)$ (resp. $[0,\tau] \times (-\infty, \e)$), in the sense of distributions.
\begin{alignat}{2}
    &\partial_t \pi_\rho = \frac{\lambda'}{\lambda} \pi_\rho + \lambda c_\rho \one_{\{\pi_\rho > 0 \}}, \quad  &&\partial_t \pi_\mu = \frac{\lambda'}{\lambda} \pi_\mu - \lambda c_\mu\one_{\{\pi_\mu > 0 \}},\label{eq:dt}\\
    &\partial_x \pi_\rho = -\lambda \one_{\{\pi_\rho > 0 \}}, \quad  &&   \partial_x \pi_\mu  = \lambda \one_{\{\pi_\mu > 0 \}},\label{eq:dx}\\
    &\partial_{xx} \pi_\rho \geq 0,  \quad && \partial_{xx}\pi_\mu \geq 0.\label{eq:dxx}
\end{alignat}
By substituting Equalities \eqref{eq:dt}, \eqref{eq:dx},\eqref{eq:dxx} into the sub-solution Inequalities \eqref{eq:subsolutionrho} and \eqref{eq:subsolutionmu}, it follows that $\pi_\rho$ and $\pi_\mu$ define pressure sub-solutions if $\lambda $ and $c_\rho,c_\mu$ satisfy the following:
\begin{alignat}{2}
&\lambda' \leq \lambda (\alpha-1) \partial_{xx} V, \quad &&\lambda' \leq \lambda (\alpha-1) \partial_{xx} W,\label{eq:lbineq}\\
&c_\rho \leq - \partial_x V, \quad &&-\partial_x W \leq c_\mu. \label{eq:cineq}
\end{alignat}
To satisfy Inequality \eqref{eq:cineq}, take 
\begin{equation}\label{eq:constants}
    c_\rho = - \sup_{[-\e,\e]} \partial_x V \text{ and } c_\mu =  - \inf_{[-\e,\e]} \partial_x W.
\end{equation}
In order to satisfy Inequality \eqref{eq:lbineq} we will choose $\lambda  = A e^{-c_{\lambda}t}$ where $A > 0$ will be fixed with respect to the initial datum and $c_\lambda> 0$ will be fixed with respect to the travelling wave construction as some constant satisfying
\begin{equation}\label{eq:clambda}
  c_{\lambda} \ge (\alpha-1) \max\{\|\partial_{xx}V\|_{L^\infty([-\e,\e])},\|\partial_{xx}W\|_{L^\infty([-\e,\e])}\}.  
\end{equation}

By our choice of $\lambda, c_\rho, c_\mu$, we have verified that Inequalities \eqref{eq:lbineq} and \eqref{eq:cineq} will be satisfied on $(-\e,\e) \times (0,\tau)$. 
Since we have only verified the inequalities on the interval $(-\e,\e)$, we restrict the time interval to any $(0,\hat{\tau})$ with 
\begin{equation}\label{eq:tauhat}
    \hat{\tau} \leq \min\left\{\frac{\e}{|c_\mu|},\frac{\e}{|c_\rho|}\right\}.
\end{equation}

In particular, for this choice of $\hat{\tau}$ the pressure variables satisfy 
\[
\pi_\rho = 0 \text{ on }[0,\hat{\tau})\times [\e,+\infty)
\text{ on } \pi_\mu = 0 \text{ on } [0,\hat{\tau})\times (-\infty,-\e]
\]
and the travelling waves are distributional solutions of the Inequalities \eqref{eq:subsolutionrho} and \eqref{eq:subsolutionmu} over the intervals $[0,\hat{\tau}] \times (-\e,+\infty)$ and $[0,\hat{\tau}] \times (-\infty,\e)$ respectively.

\medskip

Observe that we have not yet fixed the constants $A, c_\lambda$ and $\e$. 
These will be fixed in the sequel so as to be compatible with the initial datum and continuous extension.
Moreover, so long as Inequality \eqref{eq:clambda} is satisfied, we are free to choose $c_\lambda$ independently of the pair $(A,\e)$.

\medskip 

\underline{Part ii. Continuous Extension:}
So that $\pi_\rho, \pi_\mu$ define sub-solutions on the whole real line, we construct a continuous extension of the travelling wave part.
To avoid a lengthy exposition, the construction is only made for the sub-solution $\pi_{\rho}$. 
By symmetry, an analogous construction can be made for $\pi_\mu$.

\medskip 

Consider the ansatz given by  
\begin{equation}\label{eq:extension}
  \tilde{\pi}_\rho(t,x) = \left(R - W(t) - \frac{B}{2}(x-\gamma )^2\right)_{+} \text{ on } (0,\hat{\tau}) \times (-\infty, -\e) .
\end{equation}
First we connect $\tilde{\pi}_\rho$ to $\pi_\rho$ in a desirable manner, then we show that $\tilde{\pi}_\rho$ is a sub-solution for any suitably chosen triplet $W,\hat{\tau}, c_\lambda $. 
So that $\tilde{\pi}_\rho$ approaches $\pi_\rho$ continuously from the left, we fix 
\[
R = \frac{B}{2}(\gamma + \e)^2, \quad W(t) = -\lambda(t) (c_\rho t+ \e)
\]
on $(0,\hat{\tau}) \times (-\infty, -\e]$, emphasising that the constant $A,B, \gamma, \e$ and $c_\lambda$ are still left free to fix.

\medskip 

Now, so that the sub-solution profile which matches $\tilde{\pi}_\rho$ on $(0,\hat{\tau}) \times (-\infty, -\e]$ and matches $\pi_\rho$ on $(0,\hat{\tau}) \times (-\e, +\infty) $ is distributionally semi-convex at $[0,\hat{\tau}]\times \{-\e\}$ (and thus preserves the sub-solution property), $\gamma$ must be chosen to satisfy 
\[
\partial_x\tilde{\pi}_\rho(t,-\e) = B(\e+\gamma) \leq -\lambda(t) = \partial_x\pi_\rho (t,-\e).
\]
Since $-A \leq -\lambda$ for any $A$ the semi-convexity is preserved if $\gamma$ is fixed in terms of $A,B, \e$, for instance, satisfying 
\[
\gamma = - \left( \frac{A}{B}+\e\right).
\]
Now, we fix $A,B, \e$ such that the initial profiles $\tilde{\pi}_{\rho,0}$ and ${\pi}_{\rho,0}$
sit under $P_\rho$ and then choose $c_\lambda$, in terms of these constants, so that Inequality \eqref{eq:proxy} is satisfied.
In particular, consequent to the previous choice of $\gamma$ and $W$, the initial profile satisfies 
\[
\tilde{\pi}_\rho(0,x) = \left(\frac{A^2}{2B}+ A\e - \frac{B}{2}\left(x+ \left(\frac{A}{B}+\e\right)\right)^2\right)_{+} \text{ on } (-\infty,-\e).
\]
The solution $\tilde{\pi}_{\rho,0}$ is bounded by $\frac{A^2}{2B}+A\e$ and supported on the interval $[-\frac{2A}{B}- \frac{\sqrt{2A\e}}{\sqrt{B}}-\e ,-\e]$.
On the other hand, $\pi_{\rho,0}$ is bounded by $A\e $ and supported on $[-\e,0]$.
By Assumption \ref{ass:interface_2}, we know that there exists $\tilde{a}_0 > 0$ such that $P_{\rho,0} > \tilde{a}_0 $ on $(-r_0,0]$. 
Consequently, we may choose $\e^\ast > 0$ and fix the ratio $\frac{A}{B}$ such that the support of $\tilde{\pi}_{\rho,0} + \pi_{\rho,0}$ belongs to $(-r_0,0]$ for any $\e \in (0,\e^\ast)$.
Then, we may choose $A$ such that $\frac{A^2}{2B}+ A\e < \tilde{a}_0$.

\medskip 

Having ensured compatibility with the boundary data, we now prove that $\tilde{\pi}_\rho$ can be made a pressure sub-solution on the interior of its domain for a suitably small time interval. 
Consider the distributional derivatives of $\tilde{\pi}_\rho$. 
\begin{align*}
    & \partial_t\tilde{\pi}_\rho(t,x) = - W'(t)\one_{\{\tilde{\pi}_\rho >0\}}(t,x)\\
    & \partial_x \tilde{\pi}_\rho(t,x) = - B(x-\gamma)\one_{\{\tilde{\pi}_\rho >0\}}(t,x)\\
    & \partial_{xx}\tilde{\pi}_\rho(t,x) \geq -B\one_{\{\tilde{\pi}_\rho >0\}}(t,x)
\end{align*}
By substituting these derivatives within the Inequality \eqref{eq:subsolutionrho}, it follows that $\tilde{\pi}_\rho$ defines a pressure sub-solution $\text{on } [0,\hat{\tau}]\times (-\infty,-\e)$ if 
\begin{equation}\label{eq:proxy}
    (\alpha-1)(R+|W(t)|)(B + |\partial_{xx}V(x)|)\one_{\{\tilde{\pi}_\rho >0\}}(t,x) + |\partial_x V(x)|^2\one_{\{\tilde{\pi}_\rho >0\}}(t,x) \leq W'(t). 
\end{equation}

To verify that that our sub-solution can be compared with $P_\rho$, it must finally be shown that Inequality \eqref{eq:proxy} holds.
Since the support of $\tilde{\pi}_\rho$ is decreasing in time, and since $V$ is smooth, there exists $\Gamma > 0$, which depends only on $a_0$ and $r_0$, such that 
\[
 \sup_{(t,x) \in (0,\hat{\tau})\times (-\infty,-\e)}|\partial_{xx}V(x)|\one_{\{\tilde{\pi}_\rho >0\}}(t,x) \leq \Gamma, \quad   \sup_{(t,x) \in (0,\hat{\tau})\times (-\infty,-\e)}|\partial_{x}V(x)|^2 \one_{\{\tilde{\pi}_\rho >0\}}(t,x) \leq \Gamma.
\]
Then, a sufficient condition for Inequality \eqref{eq:proxy} is to show that 
\begin{equation}\label{eq:proxy2}
   C_0|W(t)| + C_1 : =(\alpha-1)(R+|W(t)|)(B+ \Gamma) +\Gamma \leqslant W'(t) 
\end{equation}

where, in particular, $C_0 = C_0(\alpha, B, \Gamma) $ and $C_1 =  C_1(\alpha,R,B,\Gamma)$.
By construction, $W'(t)$ satisfies the inequality
\begin{align*}
  W'(t)  & \geq - c_\lambda W(t) - \lambda(t) c_\rho \geq \frac{c_\lambda}{2}\left(|W(t)| + \inf_{t\in (0,\hat{\tau})} |W(t)|\right)  - |c_\rho| \sup_{t\in (0,\hat{\tau})} \lambda(t)\\
  &\geq \frac{c_\lambda}{2}|W(t)| + \frac{c_\lambda}{2}\eps A\exp(-c_\lambda \hat{\tau}) - |c_{\rho}| A.
\end{align*}
Then, for any pair $c_\lambda,\hat{\tau}>0 $ such that $c_\lambda \hat{\tau}\leq 1 $ the inequality 
\[
W'(t) \geq \frac{c_\lambda}{2}|W(t)| + \frac{c_\lambda}{2}\e A \exp(-1) -|c_{\rho}|A 
\]
holds over $(0,\hat{\tau})\times (-\infty,-\e)$.
Thus, by choosing $c_\lambda$ great enough and $\hat{\tau}$ small enough such that 
\[
c_\lambda \geq 2C_0 \text{ and } c_\lambda\geq \frac{2e}{\e A}(|c_\rho| A + C_1) 
\]
we  assert that Inequality \eqref{eq:proxy} always holds. 
In particular, such a choice may be made since $\alpha,\eps, A, B, \Gamma, R, c_\rho$ were all fixed independently of $c_\lambda$ and one may always choose a smaller $\hat{\tau}>0$ so that $c_\lambda \hat{\tau} \leq 1$ without damaging the constraint \eqref{eq:tauhat}.
\medskip 

\underline{Part iii. Mixing:}
The remainder of the proof is dedicated to showing that, for $\e \in (0,\e^\ast)$ small enough, the densities $\rho,\mu$ must mix instantaneously around the origin. 
Indeed, by the two previous steps we have constructed $\tilde{\pi}_\rho + \pi_\rho$ which acts as a pressure sub-solution with initial data less than $P_{\rho,0}$ (resp. for $\mu)$.
Consequently, it is sufficient to show that a similar mixing occurs for the sub-solution pair $(\pi_\rho,\pi_\mu)$  -- we need only study the travelling wave portion. 
We claim that such a mixing occurs when $\e$ is chosen small enough, such that:
\begin{equation}\label{eq:constantsineq}
   - \sup_{[-\e,\e]} \partial_x V =  c_\rho > c_\mu =  - \inf_{[-\e,\e]} \partial_x W.
\end{equation}
We recall that such an $\e$ always exists due to our standing Assumption \ref{ass:VW} which enforces the inequality $-\partial_xV(0) > - \partial_x W(0)$.
Moreover, the validity of this choice is detailed in the following argument.

\medskip 

Let $C_{V},C_{W}$ and $C_{V-W}$ denote the respective Lipschitz constants of $\partial_xV,\partial_x W $ and $\partial_x(V-W)$ taken over a ball $B_R(0)$ with $R \gg 1 > \e $. 
Using Equation \eqref{eq:constants} to characterise $c_\rho,c_\mu$, the following inequality is thus derived. 
\begin{align*}
c_\rho & =  - \sup_{[-\e,\e]} \partial_x V  = \inf_{[-\e,\e ]}\left( -\partial_x V\right) \geq \sup_{[-\e,\e ]} \left(-\partial_x V\right) - 2\e C_{V} \\
& = \sup_{[-\e,\e ]} \left(-\partial_x W + [\partial_x(W- V)]\right) - 2\e C_{V}  \geq \sup_{[-\e,\e ]} \left(-\partial_x W\right)  + \inf_{[-\e,\e ]}(\partial_x(W-V)) - 2\e C_{V} \\
& =  - \inf_{[-\e,\e ]} \left(\partial_x W\right)  + \inf_{[-\e,\e ]}(\partial_x(W-V)) - 2\e C_{V}  =  c_\mu + \inf_{[-\e,\e ]}(\partial_x(W-V)) - 2\e C_{V} \\
& \geq c_\mu + \partial_x(W(0)- V(0)) - \e (2C_{V} + C_{V-W}).
\end{align*}
It follows from the above system of inequalities that 
\begin{equation}\label{eq:crhocmu}
    c_\rho -c_\mu \geq \partial_x(W(0)- V(0)) - \e (2C_{V} + C_{V-W}).
\end{equation}
By Assumption \ref{ass:VW}, $\partial_x(W(0)- V(0)) > 0$. 
Consequently, $\e$ may be chosen small enough such that the right-hand side of Inequality \eqref{eq:crhocmu} is positive,  implying that $c_\rho > c_\mu$. 

\medskip 

Return to the pressure sub-solutions with $\e \in (0,\e^\ast)$ fixed such that $c_\rho > c_\mu$ (in particular, since this fixation can be made independently of $c_\lambda$, the choice does not damage the construction of the extension). 
Then, the positivity sets of the travelling wave sections of $\pi_\rho,\pi_\mu$ may be characterised as follows: 
\begin{align*}
    \pi_\rho(t,x) & > 0 \text{ for } (t,x) \in (0,\hat{\tau}) \times (-\e,c_\rho t), \\ 
    \pi_\mu(t,x) & > 0 \text{ for } (t,x) \in (0,\hat{\tau}) \times ( c_\mu t,\e).
\end{align*}
In particular, for each $ t\in (0, \hat{\tau})$, it follows that $\pi_{\rho,t},\pi_{\mu,t} > 0 $ on the interval $(c_\mu t, c_\rho  t)$ which is non-empty due to the choice $c_\mu < c_\rho$.
\end{proof}

The persistence of the interface is now concluded by contradiction. 

\begin{proposition}\label{propn:contradiction}
    Let $(\rho,\mu)$ denote a segregated solution to \eqref{eq:rho}-\eqref{eq:mu} and let $\tau > 0$ and, if $\alpha < 1$, assume that $\rho,\mu \in L^\infty([0,T]\times \R)$. 
    Then, there exists $A \subset (0,\tau)$ with $\mathscr{L}(A) > 0$ such that  
    $S(t, u_t(1)) > 0$
     for almost every $t \in A$.
\end{proposition}
\begin{proof}
    Suppose by contradiction that $S(t, u_t(1)) = 0$ for almost every $t \in (0,\tau)$.
    Then, by Proposition \ref{propn:separate}, $(\rho,\mu)$ must satisfy \eqref{eq:decoupled} on $(0,\tau)\times\R$. 
    Moreover, depending on the value of $\alpha$, Propositions \ref{propn:fast}, \ref{propn:linear} and \ref{propn:slow} assert that there exists a set $B \subset (0,\tau)\times \R$ with $\mathscr{L}^2(B) > 0$ such that $\rho,\mu > 0$ on $B$.
    This contradicts the segregation property: Point \ref{Prop1}. 
    Consequently, there must exist a set $A \subset (0,\tau)$ of positive Lebesgue measure such that $S(t, u_t(1)) > 0$ for almost every $t \in A$.
\end{proof}


\subsection{Proof of Main Theorem}
\begin{proposition}\label{propn:derivative} 
Let $(\rho,\mu)$ denote a segregated solution of System \eqref{eq:rho}--\eqref{eq:mu}. 
Then, the equality 
\[
q = \partial_x \rho_t - \frac{\rho}{S} \partial_x S_t = - \delta_{u_t(1)} S_t 
\]
is satisfied in the sense of distributions. 
\end{proposition}
\begin{proof}
By Proposition \ref{propn:segregationproperties}, the equality
\[
\rho_t(x) = S_t(x) \one_{(-\infty, u_t(1)]}(x) 
\]
is satisfied.
Thus, for a smooth test function $\varphi \in C_c^\infty([0,T]\times \R)$ the following equality holds in the sense of distributions.
\begin{align*}
& \int_0^T \int_\R \varphi\left (\partial_x \rho_t - \frac{\rho}{S} \partial_x S\right ) \diff x \diff t = - \int_0^T \int_\R \partial_x \varphi \rho_t + \varphi \frac{\rho}{S} \partial_x S  \diff x \diff t\\
& =  - \int_0^T \int_\R \left(\partial_x \varphi S_t - \varphi \partial_x S\right)\one_{\{ x\leq u_t(1)\}}   \diff x \diff t 
 =  - \int_0^T \int_{-\infty}^{u_t(1)} \partial_x \varphi S_t - \varphi \partial_x S_t  \diff x \diff t\\
 & = - \int_0^T \varphi(t,u_t(1))S_t(u_t(1)) \diff t = \int_0^T \int_\R \varphi \left(- \delta_{u_t(1)} S_t\right) \diff x \diff t . 
\end{align*}
\end{proof}
\begin{proof}[Proof of Theorem~\ref{thm:nonuniqueness-general}]
     Let $(\rho,\mu)$ denote the segregated solution to the Cauchy problem \eqref{eq:rho}-\eqref{eq:mu} as constructed in ~\cite[Theorem 3.5]{santambrogioschulz2026segregated}. 
     Since the initial data satisfies Assumption \ref{ass:positive_2}, we may further suppose $(\rho,\mu)$ that can be obtained from a vanishing viscosity approximation by the existence theory of \cite[Theorem 1.3]{Skrzeczkowski2026crossdiffusion}. 
     Then, by Proposition \ref{prop:limit-entropy} it follows that 
$$
\Phi[\rho(t),\mu(t)]\ge \int_0^t \langle q(s),\partial_x V-\partial_x W\rangle\diff  s.
$$
for almost every $ t \in [0,T]$.

\medskip 

Now, let $U$ be a neighbourhood of the origin such that $-\partial_x (V-W) > 0$  on $U$.  
By Proposition \ref{propn:localisation}, there exists $\tau_0 > 0$ such that $u_t(1) \in U$ for almost every $t \in [0,\tau_0]$. 
Consequently, the inequality 
\[
\partial_x W(u_t(1)) -
\partial_x V(u_t(1))> 0 
\text{ for }  t\in [0,\tau_0]
\]
is satisfied.
Applying Proposition \ref{propn:contradiction}, it further follows that, for every $\tau \in (0,\tau_0)$ there exists $A \subset (0,\tau)$ such that $S> 0 $ for almost every $t \in A$.
In particular, Proposition \ref{propn:contradiction} is valid in the regime $\alpha  \in (\frac{1}{3},1)$ since the result: \cite[Theorem 1.3]{Skrzeczkowski2026crossdiffusion} shows that vanishing viscosity solutions with $L^\infty(\R)$ initial datum are bounded in $ L^\infty([0,T]\times \R)$.

\medskip 

By applying Proposition \ref{propn:derivative} together with Inequality \eqref{eq:entropyineq}, we derive the positivity of the relative entropy for almost every $\tau \in (0,\tau_0)$.
\begin{align*}
\Phi[\rho(\tau),\mu(\tau)] & \geq  \int_0^\tau \langle q(s),\partial_x V-\partial_x W\rangle\diff  s =  - \int_0^\tau \partial_x (V(u_s(1))-W(u_s(1)))S(u_s(1)) \diff s \\
& \geq - \int_A \partial_x (V(u_s(1))-W(u_s(1)))S(u_s(1)) \diff s > 0.  
\end{align*}
By the above inequality, the relative entropy must be positive for almost every $\tau \in (0,\tau_0)$. 
However, the relative entropy of a segregated solution must be zero for all time. 
Consequently, the segregated solution can not coincide with the vanishing viscosity solution and there exist at least two distinct solutions to System \eqref{eq:rho}, \eqref{eq:mu}.
\end{proof}

\appendix

\section{The vanishing viscosity notion for the slow diffusion case}\label{app:porous_medium}

We detail how to modify the argument of our paper to handle the case of the slow diffusion, where we slightly modify the notion of vanishing viscosity solution, mainly~\ref{vv3}. In the paper we show that a segregated solution that has an interface is not a vanishing viscosity solution, hence proving non-uniqueness. In order to prove the result, we argue by contradiction and assume that a viscosity solution, as it is constructed in the previous work, converges towards this segregated profile. Let us note that in the segregated profile, the sum is strictly positive at the interface, and we use this and some continuity arguments to show that we only need to have convergence and~\ref{vv3} close to the interface.

\medskip 

Let $U$ be  an open interval containing the interface of the initial condition. 
In Proposition \ref{propn:localisation}, we show that, for a short time, the interface of our segregated solution remains in $U$.
Moreover, in our entropy argument we only test $q$ against $\partial_x V-\partial_x W$. 
Thus it is enough to identify $q$ where $\partial_x V-\partial_x W$ is supported.
To argue clearly, we may thus impose the condition
$$
\operatorname{supp}(\partial_x V-\partial_x W)\subset U,
\quad \partial_x W-\partial_x V>0 \quad\text{on }U.
$$

Importantly, it is possible to make this assumption without producing a circular argument relating to the size of the support of $\partial_xV -\partial_x W$ and the size of the set $U$.
This is because the size of $U$ from Proposition \ref{propn:segregationproperties} depends only on the $H^1_tL^2_x$ estimate for the function $u$ proven in \cite[Lemma 6.2]{santambrogioschulz2026segregated} which, subsequently, only depends on $\|\partial_x V\|_{L^\infty(\R)}, \|\partial_x W\|_{L^\infty(\R)}$, $\mathcal{F}[\mu_0,\rho_0]$ and the second moments of $\rho_0,\mu_0$. 
In particular, $\partial_{xx}V, \partial_{xx}W$ may be made large without changing $U$.

\medskip 

Now let $p_n=S_n^\alpha, p=S^\alpha $ be the sequences and solutions generated by the vanishing viscosity. 
It is well known that one can find a vanishing viscosity solution by adding an artificial diffusion such that $\partial_xp_n \rightharpoonup \partial_xp \quad\text{weakly in }L^2_{\mathrm{loc}}((0,T)\times\mathbb R)$
and the equation for $S_n$ gives enough compactness in time to obtain
$p_n\to p
\quad\text{strongly in }L^1_{\mathrm{loc}}((0,T)\times\mathbb R)$ and $\rho_n\to \rho$ strongly in $L^1_{\mathrm{loc}}((0,T)\times\mathbb R)$.

Then, by the previous discussion, we let $\xi(t)$ denote the segregated interface. 
Due to Proposition \ref{propn:contradiction}, we assert that
$$
S(t,\xi(t))\ge c_0>0 \quad\text{for }t\in A
$$
on a set of times $A\subset(0,T)$ with positive measure.
In fact, for the following argument, we further assume that there exists $\delta > 0 $ so that $A \subset (\delta, T)$

\medskip 

Now, recall from \cite[Theorem 1.3]{Black26} that $S_n,S$ are H\"older continuous on $(0,T)\times \R$. 
Consequently, there exists a bounded open set $Q_A \subset (\delta,T)\times \R$ with $(t,\zeta(t))\subset Q_A$ for every $t \in A$ and such that $S \geq \frac{c_0}{2}$ on $Q_A$.
Now, let $K_A \subset (\delta,T)\times \R$ be any connected compact set containing $Q_A$. 
From \cite[Corollary 1.5]{Black26}, it follows that there exists $\alpha  \in (0,1], C_\alpha > 0 $ such that 
\[
\sup_{n \in \N } \|S_n\|_{C^{0,\alpha}(K_A)} < C_\alpha
\]
and, consequently, $S_n$ converges uniformly to $S$ on $K_A$. 
Moreover, there exists $N \in \N $ such that $S_n \geqslant \frac{c_0}{3}$ on $Q_A$ when $n \geqslant N$.




\medskip 
Due to the uniform lower bound on $S_n$, we thus obtain
$$
q_n= \partial_x\rho_n-\frac{\rho_n}{S_n}\partial_xS_n \rightharpoonup
\partial_x\rho-\frac{\rho}{S}\partial_xS \quad \text{weakly in } \mathcal D'(Q_A).
$$
Moreover, for a segregated limit, we have that $q(t) =-S(t,\xi(t))\delta_{\xi(t)}$ for a.e. $t\in A$ (see Lemma~\ref{lem:dirac_interface}). 
Proposition \ref{prop:entropy-viscous} then yields for every non-negative
$\zeta\in C_c^\infty(0,T)$:
\begin{align*}
-\int_0^T \Phi[\rho_n(t),\mu_n(t)]\zeta'(t)\diff t
\ge \int_0^T \zeta(t)\langle q(t), \partial_xV-\partial_xW \rangle \diff t.
\end{align*}

Also since the limit is segregated and by properties of $\Phi$: 
$\Phi[\rho_n,\mu_n]\to 0 \quad\text{strongly in }L^1_{\mathrm{loc}}(0,T).$ Therefore, passing to the limit we obtain
\begin{equation*}
\int_0^T \zeta(t)\langle q(t),\partial_x V-\partial_x W\rangle\diff t\le0
\quad \text{for every } \zeta\ge 0,\; \, \zeta\in C_c^\infty(0,T).
\end{equation*}
Let $g(t)=\langle q(t),\partial_x V-\partial_x W\rangle$. From the previous computation we obtain $g(t)\le 0 \quad\text{for a.e. }t\in(0,T).$
However for a.e $t\in A$: 
$$
g(t) =S(t,\xi(t)) (\partial_xW-\partial_xV)(\xi(t))>0.
$$
This is a contradiction since $A$ is a set of positive measure. 
Consequently, the segregated solution is not a vanishing viscosity solution, from which we also deduce non-uniqueness.

\paragraph{Acknowledgements}
CE and GP were supported by the European Union via the ERC
AdG 101054420 EYAWKAJKOS project.
GP was also supported by the UKRI Engineering and Physical Sciences Research Council [Grant Number EP/W524426/1].
\bibliographystyle{siam}
\bibliography{biblio}

\end{document}